\documentclass[a4paper, 11pt,reqno]{amsart}
\usepackage{amsthm,amssymb,amsmath,graphicx,psfrag,enumitem,mathrsfs}
\usepackage{mathtools}
\usepackage[foot]{amsaddr}
\usepackage[british]{babel}
\usepackage[babel]{microtype}
\usepackage[margin=1in]{geometry}

\usepackage[textsize=footnotesize]{todonotes}

\usepackage{tikz}
\usepackage{setspace}

\usetikzlibrary{backgrounds}

\usepackage[numbers]{natbib}

\makeatletter
\def\NAT@spacechar{~}
\makeatother

\usepackage{hyperref}
\usepackage{thm-restate}
\usepackage[capitalise]{cleveref}
\hypersetup{colorlinks=true,
   citecolor=blue,
   filecolor=blue,
   linkcolor=blue,
   urlcolor=blue
}

\crefname{figure}{Figure}{Figures}
\Crefname{figure}{Figure}{Figures}

\allowdisplaybreaks

\newtheorem{definition}{Definition}[section]
\newtheorem{claim}{Claim}

\newtheorem{proposition}[definition]{Proposition}
\newtheorem{theorem}[definition]{Theorem}

\newtheorem{lemma}[definition]{Lemma}
\newtheorem{fact}[definition]{Fact}
\newtheorem{conjecture}[definition]{Conjecture}
\newtheorem{question}[definition]{Question}

\numberwithin{equation}{section}

\newcommand{\comment}[1]{}

\newcommand{\cP}{\mathcal{P}}

\newcommand{\cS}{\mathcal{S}}

\newcommand{\cX}{\mathcal{X}}
\newcommand{\cY}{\mathcal{Y}}

\newcommand{\ora}{\overrightarrow}

\renewcommand{\epsilon}{\varepsilon}

\newcommand{\COMMENT}[1]{}
\renewcommand{\COMMENT}[1]{\footnote{\textcolor{blue!70!black}{#1}}} 

\title{Subdivisions of digraphs in tournaments}

\author[A.~Gir\~{a}o]{Ant\'onio Gir\~{a}o}
\address[Ant\'onio Gir\~{a}o]{School of Mathematics, University of Birmingham, 
Edgbaston, Birmingham, B15 2TT, United Kingdom.
}
\email{giraoa@bham.ac.uk}

\author[K.~Popielarz]{Kamil Popielarz}
\email{kamil.popielarz@gmail.com}

\author[R.~Snyder]{Richard Snyder}
\address[Richard Snyder]{Karlsruhe Institute of Technology, Karlsruhe, Germany.}
\email{richard.snyder@kit.edu}

\thanks{The first author wishes to acknowledge support by the EPSRC, grant. no. EP/N019504/1.}

\date{\today}

\begin{document}


\maketitle

\onehalfspacing

\begin{abstract}

\setlength{\parskip}{\medskipamount}
    \setlength{\parindent}{0pt}
    \noindent
We show that for every positive integer $k$, any tournament with minimum out-degree at least $(2+o(1))k^2$ contains a subdivision of the complete directed graph on $k$ vertices, which is best possible up to a factor of $8$. This may be viewed as a directed analogue of a theorem proved by Bollob\'as and Thomason, and independently by Koml\'os and Szemer\'edi, concerning subdivisions of cliques in graphs with sufficiently high average degree. We also consider the following problem: given $k$, what is the smallest positive integer $f(k)$ such that any $f(k)$-vertex tournament contains a $1$-subdivision of the transitive tournament on $k$ vertices? We show that $f(k)= O\left (k^2\log^3 k\right)$ which is best possible up to the logarithmic factors.

\end{abstract}

\section{Introduction}\label{sec:intro}

The complete directed graph on $k$ vertices, denoted by $\ora{K}_k$, is a directed graph in which every pair of vertices is connected by an edge in each direction. As usual, we say that a tournament $T$ contains a subdivision of $\ora{K}_k$ if it contains a set $B$ of $k$ vertices and a collection of $2\binom{k}{2}$ pairwise internally vertex disjoint directed paths joining every ordered pair of vertices in $B$. We denote such a subdivision by $T\ora{K}_k$, and the vertices in $B$ are called the \emph{branch vertices} of the subdivision. Our main aim in this note is to investigate the density conditions under which a tournament must contain a $T\ora{K}_k$, where by `density' we mean specifically minimum out-degree.

The undirected analogue of this line of research has been studied extensively. The story begins with Mader~\cite{Mader1}, who showed that any graph with sufficiently large average degree contains a subdivision of the complete graph on $k$ vertices. Later, in \cite{Mader2} he showed that average degree at least $c2^k$ suffices. It is not hard to show that the random graph $G(n, 1/2)$ with high probability contains a subdivision of a clique with $\sqrt{n}/10$ vertices (and with high probability does not contain a subdivision of a clique on at least $10\sqrt{n}$ vertices). This motivated 
Mader~\cite{Mader1}, and independently Erd\H{o}s and Hajnal~\cite{ErdosHajnal}, to conjecture that any graph with average degree at least $ck^2$, for some constant $c$, should necessarily contain a subdivision of $K_k$. This was later established by Bollob\'as and Thomason~\cite{BollobasThomason}, and independently by Koml\'os and Szemer\'edi~\cite{KomlosSzemeredi}.

Given these results, it is natural to consider the problem in the directed setting, with a suitable density condition. It is not hard to see that average degree is not the right density condition: a transitive tournament has large average degree, yet clearly cannot even contain a subdivision of $\ora{K}_2$. But what about large minimum out/in-degree? This, again, does not hold, but the reason is more subtle. Indeed, Thomassen~\cite{Thomasseneven} constructed digraphs on $n$ vertices with minimum out-degree at least $c\log n$ which contain no directed cycles of even length; since any subdivision of $\ora{K}_3$ must contain an even cycle, these digraphs do not contain $T\ora{K}_k$ for any $k \ge 3$.
On the other hand, K\"{u}hn, Osthus and Young~\cite{DandD} showed that any digraph on $n$ vertices with minimum out-degree $d$ contains a subdivision of a complete digraph of order $\lfloor d^{2}/(8n^{3/2})\rfloor$, implying that any digraph on $n$ vertices with minimum out-degree $\sqrt{8k}n^{3/4}$ contains a subdivision of a complete digraph on $k$ vertices.  

The above discussion has left out the case of tournaments: Is it true that tournaments with large enough minimum out-degree contain a subdivision of the complete directed graph? The  first and last author~\cite{girao-snyder} answered this question in the affirmative: for every positive integer $k$ there is an $m(k)$ such that any tournament with minimum out-degree at least $m(k)$ contains a subdivision of the complete directed graph on $k$ vertices. This result was an important step in the proof of a partial resolution of a conjecture of Pokrovskiy~\cite{Pokrovskiytourn}.
They proved this with $m(k)$ doubly-exponential in $k^2$. Our main result in this paper is to show that we may actually take $m(k)$ to be merely quadratic in $k$. To state our main theorem, let us introduce the following function defined for every integer $k \ge 2$:
\[
    d(k) = \min\{ m:\text{any tournament } T \text{ with } \delta^{+}(T)\geq m \text{ contains a } T\overrightarrow{K}_k\}.
\]

For example, observe that $d(2) = 1$. We are able to determine $d(k)$ for all $k \ge 3$ up to a factor of $8$.
\begin{theorem}\label{thm:main}
We have that
\[
    k^2/4 \leq d(k)\leq (2+o(1))k^2, 
\]
where the $o(1)$ term goes to zero as $k \to \infty$.
\end{theorem}

The lower bound is simple: any $k^2/4$-regular tournament on $k^2/2$ vertices cannot contain a $T\ora{K}_k$ simply because such a subdivision has at least $\binom{k}{2} + k > k^2/2$ vertices. This is true, for example, of a random tournament on $k^2/2$ vertices, as such a tournament with high probability has minimum out-degree $(1 - o(1))k^2/4$. We do not know if there are better constructions, and we leave the exact determination of $d(k)$ as an open problem.

Finally, we consider a similar problem for embedding subdivisions of transitive tournaments. Recall that a tournament is \emph{transitive} if there is an ordering of the vertices such that every edge goes in the same direction.
We denote by $T_k$ the transitive tournament on $k$ vertices, and we denote by $TT_k$ any subdivision of $T_k$.
In the context of embedding subdivisions of transitive tournaments in general directed graphs, Scott~\cite{AlexScottSubdivision}, answering a question of Jagger~\cite{JaggerSubdivision}, showed that for $r \geq 2$ and $n \geq n(r)$ every directed graph on $n$ vertices with more edges than the $r$-partite Tur\'{a}n graph $T(r, n)$ contains a $TT_{r+1}$.
As for minimum degree conditions, Mader~\cite{MaderConjecture} conjectured that for all $k$ there is $f(k)$ such that any digraph with minimum outdegree $f(k)$ contains a subdivision of $T_{k}$.
This conjecture remains open to this day, even for $k = 5$.

Let $T(k)$ denote the smallest integer such that any tournament on $T(k)$ vertices contains a transitive tournament of order $k$. 
A well-known theorem of Erd\H{o}s and Moser~\cite{erdos-moser} states that $2^{(k-1)/2} \le T(k) \le 2^{k-1}$. 
In particular, any tournament on at least $2^{k-1}$ vertices contains a transitive subtournament on $k$ vertices.
If instead of finding a copy of a transitive tournament we allow each edge to be replaced by a directed path of length at most $3$, then the following result holds.
\begin{restatable}{theorem}{transsubdivision}
\label{thm:trans-subdivision}
There is a constant $C > 0$ such that the following holds. For all $k \ge 2$, any tournament on at least $Ck^2$ vertices contains a $TT_k$, where each directed path in the subdivision has length at most $3$. Moreover, this is tight up to the multiplicative constant. 
\end{restatable}

It is natural to ask if a similar lower bound on the number of vertices allows us to embed $1$-\emph{subdivisions}: subdivisions where each edge is replaced by a directed path of length $2$. 
An old conjecture of Erd\H{o}s, confirmed by Alon, Krivelevich and Sudakov~\cite{AlonKrivSudak}, states that any graph on $n$ vertices and at least $\varepsilon n^2$ edges contains a $1$-subdivision of a complete graph on $c(\varepsilon)\sqrt{n}$ vertices (in fact, they show that this holds with $c(\varepsilon) = O(\varepsilon$)). We obtain a partial directed analogue of this result, up to $\log$ factors.

\begin{restatable}{theorem}{transonesubdivision}
\label{thm:trans-1-subdivision}
Any tournament on at least $Ck^{2}\log^3 k$ vertices contains a $1$-subdivision of $T_k$.
\end{restatable}

We are able to prove this with $C = 10^7$, but no attempt is made to optimize this constant, as we believe that the same result should hold after removing the $\log$ factors (see \Cref{sec:final} for a conjecture along these lines).

\subsection{Notation and Organization}
Our notation is standard. Thus, for a vertex $v$ in a directed graph $G$, we let $N_G^+(v), N_G^-(v)$ denote the out-neighbourhood and in-neighbourhood of $v$, respectively. Moreover, we let $d_G^+(v) = |N_G^+(v)|$ denote the out-degree of $v$, and analogously $d_G^-(v)$ the in-degree of $v$. We often omit the subscript `$G$' when the underlying digraph is clear. We denote by $\delta^+(G)$ the minimum out-degree of $G$; further, if $X \subset V(G)$, we write $\delta^+(X)$ to mean the minimum out-degree of $G[X]$. For a subset $X \subset V(G)$ we let $N^+(X)$ denote the set $\bigcup_{x \in X}N^+(x)$. Lastly, if $X, Y \subset V(G)$, we write $X \rightarrow Y$ if every edge of $G$ between $X$ and $Y$ is directed from $X$ to $Y$.

The remainder of this paper is organized as follows. In \Cref{sec:main}, we prove our main theorem, \Cref{thm:main}. In fact, we shall establish a quantitative version that implies \Cref{thm:main} (see \Cref{thm:embedding}). The proof requires two preparatory lemmas, which we state and prove first. In \Cref{sec:transitive}, we establish our results \Cref{thm:trans-subdivision} and \Cref{thm:trans-1-subdivision} concerning embedding subdivisions of transitive tournaments in large enough tournaments. Finally, we conclude in \Cref{sec:final} with a further consequence of the general method of this paper, and collect a few open problems.

\section{Subdivisions of complete directed graphs}\label{sec:main}

Our aim in this section is to prove the upper bound $d(k) \leq (2 + o(1))k^2$ in \Cref{thm:main}. The proof relies on two lemmas, which we prove first. The first lemma allows us to find $k$ vertices whose in-degrees do not differ by much; such vertices will serve as the branch vertex set of our potential subdivision. Our second lemma yields a dichotomy: either we can find a partial subdivision which contains many paths of length $2$ or $3$, or we can disconnect the tournament in a particularly nice way. We first isolate the following simple fact, as it will be used elsewhere.

\begin{fact}\label{fact:in-degrees}
Let $T$ be a tournament. Then for every positive integer $\ell$ there are at most $2\ell + 1$ vertices in $T$ of in-degree (out-degree) at most $\ell$.
\end{fact}
\begin{proof}
If $L$ is the set of vertices in $T$ of in-degree at most $\ell$, then
\[
    \ell|L| \ge \sum_{v \in L} d^-(v) \ge \binom{|L|}{2},
\]
implying the bound $|L| \le 2\ell + 1$, as claimed. The proof for `in-degree' replaced by `out-degree' is identical.
\end{proof}

\begin{lemma}\label{lem:degrees}
Suppose $k \ge 3$ is an integer and let $\alpha > 0$. If $T$ is a tournament with at least $2\alpha k^2 + (20\alpha + 4)k^{7/4}$ vertices, then there exists a set $B$ of $k$ vertices and a number $m$ such that for every $v \in B$:
\begin{itemize}
    \item $d^-(v) \ge \alpha k^2 + 2k^{7/4}$.
    \item $d^-(v) \in [m - k^{7/4}, m + k^{7/4}]$.
\end{itemize}
Additionally, if $|T| = 2\alpha k^2 + (20\alpha + 4)k^{7/4}$, then $d^+(v) \le m + (20\alpha + 1)k^{7/4}$ for every $v \in B$.
\end{lemma}

\begin{proof}
By \Cref{fact:in-degrees} there must exist at least $|T| - 2\alpha k^2 - 4k^{7/4}$ vertices in $T$ whose in-degree is at least $\alpha k^2 + 2k^{7/4}$. If we partition the interval $[\alpha k^2 + 2k^{7/4}, |T|]$ into consecutive intervals of size $k^{7/4}$, then there must exist at least
\[
    k^{7/4}\cdot\frac{|T| - 2\alpha k^2 - 4k^{7/4}}{|T| - \alpha k^2 - 2k^{7/4}} \ge k
\]
vertices in the same interval. Note that the above inequality holds since it is equivalent to 
\[
    |T|(k^{3/4} - 1) \ge 2\alpha k^{11/4} + 4k^{10/4} - \alpha k^2 - 2k^{7/4},
\]
and it is not hard to verify that this is true for $k \ge 3$, using the assumption that $|T| \ge 2\alpha k^2 + (20\alpha + 4)k^{7/4}$. Finally, if $v$ is one of the $k$ vertices found above and $|T| = 2\alpha k^2 + (20\alpha + 4)k^{7/4}$, then $d^+(v) \le |T| - \alpha k^2 - 2k^{7/4} = \alpha k^2 + k^{7/4} + (20\alpha + 1)k^{7/4}$. Therefore,
\[
m \ge \alpha k^2 + k^{7/4} \ge d^+(v) - (20\alpha + 1)k^{7/4}, 
\]
establishing the last claim of the lemma.
\end{proof}

We say that a subset $B$ of vertices is $(\alpha, m, k)$-\emph{balanced} if it satisfies the two properties guaranteed by \Cref{lem:degrees}. Additionally, $T\ora{K}_k(\ell_1, \ell_2)$ denotes a partial subdivision of $\ora{K}_k$ with precisely $\ell_1$ paths of length $2$, $\ell_2$ paths of length $3$, and no paths of length greater than $3$. If $U \subset V(T)$ disconnects $T$, then $T\setminus U$ decomposes as $S \cup T'$ where $S \cap T' = \varnothing$, $S, T' \neq \varnothing$, and $S \rightarrow T'$. In this situation, we call $S$ the \emph{source component}, and $T'$ the \emph{sink}. The following key lemma says that either we can find a suitable $T\ora{K}_k(\ell_1, \ell_2)$, or there exists a subset $U$ of vertices which disconnects $T$, and such that the source component of the remaining tournament is quite large.

\begin{lemma}\label{lem:dichotomy}
Suppose $k \ge 3$ is an integer, $T$ is a tournament with $\delta^+(T) \ge k^2 + 2k^{7/4}$, and suppose $B \subset V(T)$ is an $(\alpha, m, k)$-balanced subset of $k$ vertices for some $\alpha, m > 0$. Then one of the following must occur:
\begin{enumerate}
    \item \label{itm:dichotomy1} There is a copy of $T\ora{K}_k(\ell_1, \ell_2)$ in $T$ with branch vertex set $B$ such that
        \[
        4(\ell_1 + \ell_2) + 6k^{7/4} > m.
        \]
    \item There is a subset $U \subset V(T)$ that disconnects $T$ such that the source component $S$ of $T \setminus U$ satisfies $|S| \ge |U|+k$. Moreover, the sink $T\setminus (U \cup S)$ has size at least $k$.
\end{enumerate}
\end{lemma}

\begin{proof}
Let $B$ be an $(\alpha, m, k)$-balanced $k$-set of vertices in $T$, and suppose ($1$) fails in the statement of the lemma. List the edges $e_1, \ldots, e_N$ of $T^*[B]$ where $N := \binom{k}{2}$, and $T^*$ is the tournament obtained from $T$ by reversing $T$'s edges. Then for any permutation $\sigma: [N] \to [N]$ there is an index $f = f(\sigma)$ such that the edges $e_{\sigma(1)}, \ldots, e_{\sigma(f-1)}$ can be successfully embedded as paths of length $2$ or $3$, but $e_{\sigma(f)}$ cannot, and the resulting copy of $T\ora{K}_k(\ell_1, \ell_2)$ satisfies
\[
m \ge 4(\ell_1 + \ell_2) + 6k^{7/4}.
\]
Pick an ordering $\sigma$ such that the number of paths $\ell_1$ of length $2$ in the partial subdivision $\cS$ with embedded edges $e_{\sigma(1)}, \ldots, e_{\sigma(f-1)}$ is maximized. Without loss of generality, we may assume $\sigma$ is the identity permutation, and let $e_f = xy$. Since $e_f$ fails to embed we must have that every edge is directed from $N^-(y) \setminus V(\cS)$ to $N^+(x) \setminus V(\cS)$. Similarly, we have that $W = N^+(x) \cap N^-(y) \subset V(\cS)$. Let $A$ denote the set of $\ell_1$ non-branch vertices that are on paths of length $2$ in $\cS$. We claim that, in fact, $W \subset A \cup B$. Indeed, if there is $e_i = uv$, $i < f$, and a $u-v$ path $uzwv$ with, say, $z \in W$, then consider the embedding order where we swap $e_i$ and $e_f$ and embed $e_f$ using the $2$-path $xzy$. This is a legal embedding of $e_f$, as $z$ does not belong to any of the subdivided edges $e_{j}$ with $j \neq i, f$. But now we have an embedding order with more directed paths of length $2$ in the partial subdivision, contradicting our choice of $\sigma$.

 Now since $B$ is $(\alpha, m, k)$-balanced, all in-degrees differ by at most $k^{7/4}$, and so
\begin{align*}
   |N^-(x) \setminus N^-(y)| &\le |N^-(y) \setminus N^-(x)| + k^{7/4} \\
   & = |W| + k^{7/4}.\\
   &\le \ell_1+ k + k^{7/4}.
\end{align*}
Note that the last inequality holds since $W \subset A \cup B$.
Now, let $U = V(\cS) \cup (N^-(x) \setminus N^-(y))$ and observe that $|U|$ satisfies the upper bound
\begin{align*}
|U| &\le 2\ell_2 + \ell_1 +  k + (\ell_1 + k + k^{7/4})\\
&= 2(\ell_1 + \ell_2) + 2k + k^{7/4}\\
\end{align*}
Then $T \setminus U$ is disconnected with source component $S = N^-(y)\setminus U$ and sink $N^{+}(x)\setminus V(\cS)$. By the minimum out-degree condition on $T$, the sink has at least $k^2 + 2k^{7/4} - (2\binom{k}{2} + k) > k$ vertices. Finally, as $|S| \ge |N^-(y)| - |U|$, and recalling that $|N^-(y)| \ge m - k^{7/4}$ and $m \ge 4(\ell_1 + \ell_2) + 6k^{7/4}$, we have
\begin{align*}
    |S| &\ge (m - k^{7/4}) - (2(\ell_1 + \ell_2) + 2k + k^{7/4} )\\
    &\ge 2(\ell_1 + \ell_2) + 4k^{7/4} - 2k\\
    &\ge |U| + (3k^{7/4} - 4k)\\
    &\ge |U| + k,
\end{align*}
where the last inequality follows since $k \ge 3$. This completes the proof
of the lemma.
\end{proof}

With the above lemma complete, we are ready to prove the main result of this section.

\begin{theorem}\label{thm:embedding}
Let $k \ge 2$ be an integer. Any tournament $T$ with $\delta^{+}(T)\geq 2k^2+147k^{7/4}$ contains a $T\ora{K}_k$. Moreover, this subdivision has the property that each edge is subdivided at most twice.
\end{theorem}

The rough idea of the proof is as follows. We shall iteratively apply \Cref{lem:dichotomy}: if at some point we find a partial subdivision with many paths of length $2$ or $3$, then we stop. Otherwise, we obtain a cut set $U$ with source component $S$ satisfying $|S| \ge |U|$, and look to apply the lemma again to $T \setminus (U \cup S)$. Eventually we either obtain a partial subdivision $\cS$, or reach a subtournament $T'$ that is quite small. In the first case, we show how to extend this partial subdivision to a full subdivision. In the latter case, since $T'$ is small and the minimum out-degree is large, every vertex in $T'$ has many out-neighbours outside of $T'$. We use this fact, together with some structural features of the cut sets and source components, to embed the requisite paths. Let us now make these ideas more precise.

\begin{proof}[Proof of \Cref{thm:embedding}]
We may assume that $k \ge 3$ as any tournament with $\delta^+(T) \ge 1$ contains a directed cycle (i.e., a subdivision of $\ora{K}_2$).
We shall apply \Cref{lem:dichotomy} repeatedly to obtain subtournaments $T_1 := T, T_2, \ldots$, and subsets $U_0 := \varnothing, U_1, U_2, \ldots$, such that for every $i \ge 1$
\begin{itemize}
    \item $U_i \subset T_i$.
    \item $|T_i| \ge k$.
    \item $T_i \subset T_j$ if $i > j$.
    \item $T_i \setminus U_i$ is disconnected with source component $S_i$ satisfying $|S_i| \ge |U_i|$ and such that $|T_i \setminus (U_i \cup S_i)| \ge k$.
\end{itemize}
Indeed, set $T_1 = T$, $U_0 = \varnothing$ and suppose $T_i, U_{i-1}$ have already been defined for some $i \ge 1$. We shall show how to obtain $T_{i+1}$ and $U_{i}$ as follows. We claim that either $T_i$ contains a subtournament $T'_{i}$ which has large minimum out-degree, or we can find $k$ vertices of small minimum out-degree in $T_i$. To make this precise, initialize $R = \varnothing$. If there is a vertex $v$ in $T'_{i} = T_i$ with $d^+_{T'_{i}}(v) < k^2 + 12k^{7/4}$, then add it to the set $R$. Looking at $T_i' = T_i\setminus\{x\}$, we repeat the same process to $T_{i}\setminus\{x\}$ and so on. Either we obtain $|R| = k$ or $T'_i=(T_i\setminus R) \neq \emptyset$ with $|R| < k$ (as $|T_i| \ge k$) and $\delta^+(T'_{i}) \ge k^2 + 12k^{7/4}$. As $\delta^+(T'_i) \ge k^2 + 12k^{7/4}$ we easily have $|T'_i| \ge 2k^2 + 24k^{7/4}$. Choose a real number $\alpha \ge 1$ such that
\[
|T_i| = 2\alpha k^2 + (20\alpha + 4)k^{7/4},
\]
and apply \Cref{lem:degrees} to $T_i$. We obtain an $(\alpha, m, k)$-balanced subset $B_i \subset V(T'_i)$ of $k$ vertices, for some $m$. Now apply \Cref{lem:dichotomy} to $T'_i$ and $B_i$. If condition ($1$) holds from the lemma, then we terminate the procedure at step $i$ and obtain a partial subdivision $\cS = T\ora{K}_k(\ell_1, \ell_2)$ in $T_i$ on branch vertex set $B_i$ satisfying
\[
4(\ell_1 + \ell_2) + 6k^{7/4} > m.
\]
Otherwise, ($2$) holds. Let $U'_i$ be the cut set and $S'_i$ the source component, and moreover, let $C_i=U'_i\cup R$. Note that by ($2$) of \Cref{lem:dichotomy} we have $|S'_i|\ge |U'_i| + k \ge |C_i|$, and the sink, namely $T_i\setminus (C_i\cup S'_i)$, has size at least $k$. 

It follows that we may choose a set $U_{i}\subset T_i$ of minimum possible size such that $T_i \setminus U_i$ is disconnected with source component $S_i$ satisfying $|S_i| \ge |U_i|$, and such that the sink $T_{i+1}$ has size at least $k$. 
We can continue applying the same argument to $T_{i+1}$. Note that eventually this process must terminate. Indeed, for each $i$ we have that $|T_{i+1}| < |T_i|$ (as $|U_i \cup S_i| \ge 1$). So we must reach a stage $t$ where for which $T_{t+1}$ either contains a partial subdivision $\cS$ as per ($1$) of \Cref{lem:dichotomy}, or we find $k$ vertices all of which have less than $k^2 + 12k^{7/4}$ out-neighbours in $T_{t+1}$. Thus we have established the following:

\begin{claim}\label{claim:result}
The above procedure terminates at some stage $t \ge 1$ with either
\begin{enumerate}
    \item \label{itm:result-1} a partial subdivision $\cS = T\ora{K}_k(\ell_1, \ell_2) \subset T_{t+1}$ satisfying $4(\ell_1 + \ell_2) + 6k^{7/4} > m$, or
    \item \label{itm:result-2} a subset $B \subset T_{t+1}$ of $k$ vertices with $d^+_{T_{t+1}}(v) < k^2 + 12k^{7/4}$ for every $v \in B$. \hfill\qedsymbol
\end{enumerate}
\end{claim}

We shall denote by $B$ either the branch vertex set of $\cS$ in the first case of \Cref{claim:result}, or the $k$-set obtained in the second case. This set $B$ will play the role of the branch vertex set of the full subdivision we wish to embed. Let $U = \bigcup_{i=1}^{t}U_i$ and $S = \bigcup_{i=1}^tS_i$. The following claim asserts that for each $i \in [t]$ all subsets of $U_i$ send many out-edges to $S_i$.

\begin{claim}\label{claim:expansion}
For each $i \in [t]$ the following holds: for every non-empty subset $X \subseteq U_i$
\[
    |N^+(X) \cap S_i| \ge |X|/2.
\]
\end{claim}
\begin{proof}
Let $X$ be a non-empty subset of $U_i$. If $S_i \subset N^+(X)$, then 
\[
|N^+(X)\cap S_i| = |S_i| \ge |U_i| \ge |X|,
\]
where the first inequality holds by definition of the sets $U_i$ and $S_i$. So we may assume that $S_i$ is not contained in $N^+(X)$. 
If the claim is false, then replace $U_i$ by $U'_i = (U_i \setminus X) \cup (N^+(X)\cap S_i)$ and set $S'_i = S_i \setminus N^+(X)$. The set $U'_i$ has size strictly smaller than the size of $U_i$ and still disconnects $T_i$. Moreover,
\begin{align*}
    |S_i'| &= |S_i| - |N^+(X) \cap S_i|\\
    &\ge |U_i| - |X|/2\\
    &\ge |U_i| - |X| + |N^+(X)\cap S_i| = |U'_i|,
\end{align*}
which contradicts the minimal choice of $U_i$. This completes the proof of the claim.
\end{proof}

The next lemma asserts that, as long as vertices in $B$ send enough out-neighbours outside of $T_{t+1}$, then we may embed the required internally vertex disjoint directed paths joining prescribed pairs in $B$.

\begin{lemma}\label{lem:outnbrs-embedding}
Let $\ell \ge 1$ be an integer and let $(x_1, y_1), \ldots , (x_\ell, y_\ell)$ be distinct pairs of vertices in $B$ with $x_i \neq y_i$ for each $i \in [\ell]$. If every vertex in $B$ has at least $2\ell$ out-neighbours in $T\setminus T_{t+1}$, then there exist pairwise internally disjoint directed paths $P_i$ of length 3 joining $x_i$ to $y_i$ for every $i \in [\ell]$.
\end{lemma}

Before proving the lemma, we record the following simple consequence of Hall's theorem that we need.

\begin{proposition}\label{prop:Hall}
Suppose $G$ is a bipartite graph with vertex sets $U, V$ such that $|N(X)| \ge |X|/2$ for every $X \subset U$. Then there is a set $M \subset E(G)$ with the property that every vertex in $U$ is incident to exactly one edge in $M$, and every vertex in $V$ is incident to at most two edges of $M$.
\end{proposition}
\begin{proof}
For every $v \in V$ add a new vertex $v'$ and join $v'$ to all of $v$'s neighbours; call the resulting graph $G'$. Then for every $X \subset U$ we have $|N_{G'}(X)| \ge |X|$, so by Hall's theorem there is a matching of $U$ in $G'$. The result follows by identifying vertices in $V$ with their duplicates.
\end{proof}

\begin{proof}[Proof of \Cref{lem:outnbrs-embedding}]
Combining \Cref{claim:expansion} with \Cref{prop:Hall}, it follows that for each $i \in [t]$ there is a partition $U_i = U_i'\cup U_i''$ such that $U_i'$ and $U_i''$ are both matched into $S_i$. Let $U' = \bigcup_{i=1}^t U_i'$ and $U'' = \bigcup_{i=1}^t U_i''$ so that $U = U' \cup U''$, and fix a directed matching $M'$ from $U'$ to $S$, and a directed matching $M''$ from $U''$ to $S$. Additionally, for each $i \in [\ell]$ let $N_i$ denote the out-neighbourhood of $x_i$ in $T\setminus T_{t+1}$. Observe that some of these $N_i$'s may repeat (as some of the $x_i$'s may repeat among the $\ell$ pairs). Also, since $S \rightarrow T_{t+1}$,
\[
    |N_i| = |N^+(x_i) \cap U| \ge 2\ell,
\]
for each $i = 1, \ldots, \ell$. Let $X' \subset [\ell]$ be those indices $i$ for which $|N_i \cap U'| \ge \ell$ and $X'' = [\ell]\setminus X'$. Note that by our assumption that each $x_i$ has at least $2\ell$ out-neighbours outside of $T_{t+1}$, it follows that $|N_i\cap U''| > \ell$ for each $i \in X''$. Now,  we may pick $|X'| \le \ell$ distinct vertices in $U'$ such that each vertex is an out-neighbour of one of the $x_i$'s with $i \in X'$. Thus we obtain a collection $\cP$ of directed paths of length $3$ by using the appropriate matching edge from $M'$, and the fact that $S \rightarrow \{y_1, \ldots, y_\ell\}$. It remains to find the analogous directed paths joining $(x_i, y_i)$ for $ i \in X''$. Let $A$ denote the set of $|X'|$ vertices in $S$ used in paths in $\cP$. Remove from $U''$ every vertex which is matched by $M''$ to a vertex of $A$; obviously we remove at most $|X'|$ vertices, so each $x_i$ with $i \in X''$ has more than $\ell - |X'| = |X''|$ suitable out-neighbours left in $U''$. Therefore we can pick $|X''|$ distinct vertices in $U''$ with the property that each such vertex is an out-neighbour of one of these $x_i$'s, and use the appropriate matching edges from $M''$ (avoiding $A$) as before. Hence, we have found paths joining all pairs, completing the proof.
\end{proof}

Now we are ready to complete the proof of \Cref{thm:embedding}. Suppose first that we are in Case~(\ref{itm:result-2}) of \Cref{claim:result}; that is, $d_{T_{t+1}}^+(v) < k^2 + 12k^{7/4}$ for every $v \in B$. Then by the minimum out-degree condition on $T$, we have that each vertex in $B$ has more than $k^2$ out-neighbours outside of $T_{t+1}$, and as $k^2 > 2\binom{k}{2}$, \Cref{lem:outnbrs-embedding} implies that we can embed all the required $\binom{k}{2}$ paths.

So we may assume that we are in Case~(\ref{itm:result-1}). Then we have a partial subdivision $\cS = T\ora{K}_k(\ell_1, \ell_2)$ on branch vertex set $B$, where $B$ is $(\alpha, m, k)$-balanced for some $\alpha \ge 1$. As $4(\ell_1 + \ell_2) + 6k^{7/4} > m$, one has
\[
\alpha k^2 + k^{7/4} \le m < 4\binom{k}{2} + 6k^{7/4}, 
\]
so crudely we have $\alpha \le 7$. As we need to embed $\binom{k}{2} - \ell_1 - \ell_2$ more paths, in view of \Cref{lem:outnbrs-embedding} and the minimum out-degree condition $\delta^+(T) \ge 2k^2 + 147k^{7/4}$, we are done provided
\[
    2k^2 + 147k^{7/4} - d^+_{T_{t+1}}(v) \ge 2\left(\binom{k}{2} - \ell_1 -\ell_2\right)
\]
holds for every $v \in B$. But this is true since by \Cref{lem:degrees} we have $d_{T_{t+1}}^+(v) \le m+(20\alpha + 1)k^{7/4} \le  m + 141k^{7/4}$ for every $v \in B$, and so 
\begin{align*}
    2\binom{k}{2} - 2(\ell_1 + \ell_2) + (m + 141k^{7/4}) &< k^2 + 2(\ell_1 + \ell_2) + 147k^{7/4} \\
    &\le 2k^2 + 147k^{7/4}, 
\end{align*}
where the first inequality follows using the bound $m < 4(\ell_1 + \ell_2) + 6k^{7/4}$, and the last inequality holds since always $\ell_1 + \ell_2 \le \binom{k}{2}$. Thus we may embed all remaining paths yielding a $T\ora{K}_k$ in $T$.
\end{proof}

Observe that our proof shows that we can embed a $T\ora{K}_k$ where each path in the subdivision has length at most $3$. We remark that this is best possible in the sense that there exist tournaments with large minimum out-degree which cannot contain copies of $T\ora{K}_k$ where each path has length at most $2$. For example, it is routine to check that a blow-up of a cyclic triangle where each class is a copy of the transitive tournament on $10k^2$ vertices has this property.

\section{Subdivisions of transitive tournaments}\label{sec:transitive}


Our aim in this section is to prove \Cref{thm:trans-subdivision} and \Cref{thm:trans-1-subdivision}. We begin with a lemma similar in spirit to \Cref{lem:degrees}, but which is tailored to our specific needs in this section. To state it, we say that a subset $B \subset V(T)$ is $C$-\emph{nearly-regular} if either $d^-(v) \le d^+(v) \le Cd^-(v)$ for every $v \in B$, or $d^+(v) \le d^-(v) \le Cd^+(v)$ for every $v \in B$. Further, $B$ is $(C, m, k)$-\emph{nearly-regular} if it is $C$-nearly-regular and additionally $d^-(v) \in [m - 10k, m + 10k]$ for every $v \in B$. The following lemma allows us to find $(4, m, k)$-nearly-regular $k$-element subsets in tournaments.

\begin{lemma}\label{lem:ratios}
Any tournament $T$ contains a $4$-nearly-regular subset of size $|T|/10$, and a $(4, m, k)$-nearly-regular subset of size $k$, for some $m$.
 \end{lemma}
 
\begin{proof}
We first claim that $T$ contains a $4$-nearly-regular subset of size at least $|T|/10$. Indeed, let $|T| = n$ and let $R \subset V(T)$ be the vertices for which either the ratio between the out-neighborhood and in-neighbourhood (or vice-versa) is between $1$ and $4$. If $|R|\geq n/5$, then we are done, as we may pass to a subset $A \subset R$ of at least half the size for which the property is satisfied for one or the other. 
If not, then let $T'=T\setminus R$, so that $|T'|\geq 4n/5$. Let $T'_1$ be the set of vertices $v \in V(T')$ for which $d_T^+(v) > 4d_T^-(v)$ and $T'_2$ be those vertices $v \in V(T')$ for which $d_T^-(v) > 4d_T^+(v)$. Suppose without loss of generality that $|T'_1|\geq |T'_2|$, so that $|T'_1|\geq 2n/5$. This implies that there is a vertex $u$ in $T'_1$ which has in-degree inside $T'_1$ at least $n/5$. But then
\[
    n/5 \leq d_T^-(u) < \frac{1}{4}d_T^+(u) \le n/5,
\]
a contradiction.

Thus, we can always find a $4$-nearly-regular subset $A$ of size at least $|T| / 10$. As in the proof of \Cref{lem:degrees}, partition the interval $[1, \ldots, |T|]$ into consecutive intervals of size $10k$, and distribute the vertices of $A$ according to their in-degrees in $T$. By the pigeonhole principle, there must exist at least
\[
    10k\cdot \frac{|A|}{|T|} \ge 10k\cdot \frac{1}{10} = k
\]
vertices in the same interval. These $k$ vertices form a $(4, m, k)$-nearly-regular subset for some $m$.
\end{proof}

We are now in a position to prove \Cref{thm:trans-subdivision}, which we restate here for convenience. The proof is not very different from that of \Cref{lem:dichotomy}: either we can find what we are looking for, or we can pass to a `nice' subtournament which allows us to embed the required subdivision by induction.
\transsubdivision*
\begin{proof}
First, observe that with high probability a uniformly random tournament $T$ on $k^2/10$ vertices does not induce a set of size $k$ whose distance to a transitive tournament is smaller than $k^2/6$. This implies $T$ can not contain a $TT_k$ since any such subdivision must span at least $k^2/6$ vertices.
Let $C=150$. 
We shall apply induction on $k$. 
For $k=2$, the statement holds trivially since any tournament with at least $2$ vertices contains a subdivision of a transitive tournament on $2$ vertices. 
Suppose we want to prove the statement for $k$. 
Let $T$ be a tournament on $Ck^2$ vertices.
Applying Lemma~\ref{lem:ratios} to $T$, we obtain a $(4,m,k)$-nearly-regular set $S\subset T$ consisting of $k$ vertices. Without loss of generality, assume that $d^-(v) \le d^+(v) \le 4d^-(v)$ holds for every $v \in S$. We shall iteratively try to embed a subdivision on these $k$ branch vertices. Observe that we may always choose an ordering $\sigma$ of $S$ for which we just need to embed $\binom{k}{2}/2$ extra paths to find a transitive subdivision. Suppose we are at step $i < \binom{k}{2}$ and we have already found $i$ paths of the subdivision; we may assume $i$ is maximal. Let $P\subset T$ be the set consisting of the inner vertices of the paths already found. Note that $|P|\leq 2\cdot k^2/4$, since each path we have embedded has at most $2$ inner vertices. 
Suppose now we want to find a directed path from $x$ to $y$ (where $x$ lies before $y$ in the ordering $\sigma$ of $S$). By the maximality of $i$, we must have $N^{+}(x)\cap N^{-}(y)\subset P \cup S$, so $|N^+(x) \cap N^-(y)| \le k^2/2 + k \le k^2$. Furthermore, since $S$ is $(4, m,k)$-nearly-regular,
\begin{align*}
    |N^{-}(x)\cap N^{+}(y)| &\le |N^+(x) \cap N^-(y)| + 10k\\
    &\le k^2 + 10k\\
    & \le 4k^2.
\end{align*}

Delete the set $\left (N^{+}(x)\cap N^{-}(y)\right) \cup \left (N^{-}(x)\cap N^{+}(y) \right) $ from $T$ and denote by $T'$ the remaining tournament. Then $T'$ splits into two disjoint sets $A,B$ where $A$ is the common out-neighborhood of $x,y$, and $B$ is the common in-neighborhood of $x,y$.
We claim that the partition $V(T') = A \cup B$ satisfies the following two properties:
\begin{enumerate}
    \item $ |A| + |B| \ge (1 - 1/30)|T|$.
    \item $\min\{ |A|, |B| \} \ge |T|/6$.
\end{enumerate}

To see the the first property, simply observe that we have removed at most $5k^2$ vertices to obtain $T'$, and therefore $|A| + |B| \ge (1 - 1/30)150k^2 = (1 - 1/30)|T|$. To see the second property, since $S$ is $(4, m, k)$-nearly-regular, for every $v \in S$
\[
    d^-(v) \ge d^+(v)/4 = (|T| - d^-(v))/4,
\]
implying $d^-(v) \ge |T|/5$. Then also $d^+(v) \ge d^-(v) \ge |T|/5$. It follows that $\min\{|A|, |B|\} \ge |T|/5 - 5k^2 \ge |T|/6$ (using our choice of $C$), as claimed

Without loss of generality suppose that $|A| \le |B|$. 
 If there is a directed edge from $A$ to $B$, we may find a directed path from $x$ to $y$ of length $3$, which contradicts the maximality of $i$. Accordingly, $B \rightarrow A$. Since $|A| \ge |T|/6 = Ck^2/6 \ge C(2k/5)^2$,
the induction hypothesis guarantees a subdivision of a transitive tournament on $2k/5$ vertices in $T[A]$, where each path has length at most $3$. Similarly, $T[B]$ contains a subdivision of a transitive tournament on $3k/5$ vertices, because $|B|\geq (1/2-1/60)Ck^2\geq C(3k/5)^2$. As $B \rightarrow A$, these two subdivisions may be put together to form a subdivision of a transitive tournament on $k$ vertices where each path has length at most $3$.
\end{proof}

We close this section by proving \Cref{thm:trans-1-subdivision} (recall that a $1$-subdivision of $T_k$ is a subdivision of the transitive tournament of order $k$ where each directed path has length $2$).
\transonesubdivision*

Before proving this theorem we need a lemma. Given a graph $G$ and a vertex $x \in V(G)$ we denote by $B_r(x) = \{v \in V(G): d_G(x, v) \le r \}$ the ball of radius $r$ in $G$ around $x$. The following states that if $G$ has the property that every ball of radius $C\log^2 n$ is small, then $G$ can be disconnected by $o(n)$ vertices into many small components. Recall that $\log n$ denotes the logarithm of $n$ to the base $e$.
\begin{lemma}\label{lem:small-balls}
Let $G$ be an $n$-vertex graph with the property that for any $r \le 10\log^2 n$ and any vertex $x$ we have $|B_r(x)| \le \frac{n}{5\log n}$. Then $G$ contains a set $S\subset V(G)$ of size at most $\frac{n}{5\log n}$ such that $G-S$ is the union of connected components each of which has size at most $\frac{n}{5\log n}$.  
\end{lemma}
\begin{proof}
Pick a vertex $x\in V(G)$ and perform a breadth-first-search from $x$, obtaining levels $L_0, L_1, \ldots , L_n$ such that $L_i = \{v \in V(G): d(x, v) = i\}$. Denote by $k=10\log^2 n$.
We claim that there is a $k' < 10\log^2 n$ for which $|L_{k'}| <|B_{k'-1}(x)|/(5\log n)$. If not, for any $k'< 10\log^2 n$, we have 
\[
    |B_{k'}(x)| \ge \left (1 + \frac{1}{5\log n}\right)|B_{k'-1}(x)|,
\]
and hence by induction $|B_k(x)| \ge \left(1 + \frac{1}{5\log n}\right)^k$. Using the inequality $1 + x \ge e^{\frac{x}{x+1}}$ (valid for any $x > - 1$) with $x = 1 + 1/5\log n$ we obtain
\[
    1 + \frac{1}{5\log n} \ge e^{\frac{1}{1 + 5\log n}} \ge e^\frac{1}{10\log n}.
\]
Hence, since $k = 10 \log^2 n$, this yields $|B_k(x)| > (e^{\frac{1}{10\log n}})^{10\log^2 n}>n$, which is clearly a contradiction.

So we may choose $k' < k$ such that $|L_{k'}| < \frac{|B_{k'-1}(x)|}{5\log n}$; remove $L_k$ from $G$. Pick a vertex from each connected component of size larger than $\frac{n}{5\log n}$, and perform the same procedure as above. Eventually this process must terminate, and the components we are left with are all balls of radius less than $10\log^2 n$, so by assumption have at most $\frac{n}{5\log n}$ vertices. Moreover, the union of the sets removed $S$ has size $|S| \le \frac{n}{5\log n}$ by construction, completing the proof of the lemma.
\end{proof}

We are now in a position to prove \Cref{thm:trans-1-subdivision}. The proof goes roughly as follows. We define an auxiliary graph on $V(T)$ where $x \sim y$ if and only if the symmetric difference of their out-neighbourhoods has size less than $2k^2$ (it is helpful to think of this as being `bad' for embedding $1$-subdivisions since, roughly speaking, this implies that $|N^+(x)\setminus N^+(y)| = |N^+(x) \cap N^-(y)|$ is `small'). It turns out that $G$ satisfies the properties required to apply \Cref{lem:small-balls}. So $G$ splits into many small components, and therefore every pair of vertices $x, y$ in different components have $|N^+(x) \Delta N^+(y)| \ge 2k^2$. Moreover, if $d^+(x) \ge d^+(y)$, then it is not hard to show that actually $|N^+(x) \cap N^-(y)| \geq \binom{k}{2}+k$. So order the vertices of the components according to non-increasing out-degree. Finally, we show that enough vertices from the components intersect the first half and second half of the order, enough that we may apply induction to embed a $1$-subdivision of $T_{k/2}$ in each half. Then we can greedily embed the remaining directed paths of length $2$ between these partial $1$-subdivisions.
\begin{proof}[Proof of \Cref{thm:trans-1-subdivision}]

The proof will be by induction on $k$ with $C = 10^7$. For $k\le 3$ the theorem follows: $T$ contains a transitive tournament on at least $\log_2|T| > 6$ vertices, which contains a $1$-subdivision of $T_3$. 
So let $k > 3$ and let $T$ be a tournament with $|T| := n = Ck^{2}\log^3 k$ vertices. Construct an auxiliary graph $G$ on $V(T)$ in the following way: join $x$ to $y$ if $|N^+(x) \Delta N^+(y)| < 2k^2$.
Now, apply \Cref{lem:small-balls} to $G$. To see that $G$ satisfies the property needed for the lemma, suppose there is a vertex $x$ which sees at least $\frac{n}{5\log n}$ vertices in the ball $B_r(x)$ of radius $r = 10\log^2 n$. It is not hard to check that $\log n \le 20\log k$. Now, there is a path in $G$ of length at most $20\log^2 n \le 8000\log^2 k$ between every pair of vertices in $B$. It follows that every such pair has the property that the symmetric difference between their out-neighborhoods is at most $16000k^2\log^2 k$, by the definition of $G$. But this is impossible because $B_r(x)$ has order at least $n/5\log n \ge 10^5k^2\log^2 k$. By looking at the tournament $T' = T[B_r(x)]$, we observe that it contains a vertex $y$ whose out-neighborhood $N_{T'}^{+}(y)$ has size at least $(10^5k^2\log^2 k) / 2 = 50000k^2 \log^2 k$ , and by the same reasoning, inside the tournament induced on $N_{T'}^{+}(y)$ there must exist a vertex $z$ whose in-neighborhood has size at least $25000k^2\log^2 k$. Accordingly, in $T$ we have
\[
    |N^+(y)\Delta N^+(z)| \ge 25000k^2\log^2 k.
\]
But this contradicts the fact that we must have $|N^+(y)\Delta N^+(z)| \le 16000k^2\log^2 k$, as we established earlier.

So we may apply \Cref{lem:small-balls} to $G$. This yields a set $S$ of vertices such that $|S| \le \frac{n}{5\log n}$ and $G' = G - S$ consists of connected components $C_1,C_2,\ldots,C_t$. 
We claim that for any two vertices belonging to different components, there are many directed paths of length two between them.
\begin{claim} \label{claim:many_paths}
Let $u, v$ be any two vertices belonging to different components such that $d^+(u) \geq d^+(v)$.
Then $|N^+(u) \cap N^-(v)| \geq k^2 - 1 \geq \binom{k}{2}+k$.
\end{claim}
\begin{proof}

Observe first that as $u,v$ belong to different components, they satisfy 
\[
  |N^+(u) \cap N^-(v)| + |N^+(v) \cap N^-(u)| = |N^+(u) \Delta N^+(v)| - 1 \geq 2k^2 - 1.
\]
Moreover, since $d^+(u) \geq d^+(v)$ we have that $|N^+(u) \cap N^-(v)| \geq |N^+(v) \cap N^-(u)| - 1$.
Therefore $|N^+(u) \cap N^-(v)| \geq (2k^2-2)/2 = k^2 - 1$, as claimed.
\end{proof}
 
Let $\sigma$ be an ordering of the vertices in $V(G')$ so that their out-degrees are non-increasing.
Furthermore, let 
$m = |V(G')| = |V(G - S)| \ge (1 - \frac{1}{5\log n})n$. We assume that $m = (1 - \frac{1}{5\log n})n$ by possibly removing some vertices from $G'$. Let $A_1$ denote the initial segment (according to $\sigma$) of $V(G')$ with $\lfloor{m/2}\rfloor$ vertices, and $A_2$ the remaining $\lceil{m/2}\rceil$ vertices. The following lemma allows us to partition the components $C_1, \ldots, C_t$ into two families $\cX$ and $\cY$ such that the components in $\cX$ intersect $A_1$ in a set $X$ of `many' vertices, and the components in $\cY$ intersect $A_2$ in a set $Y$ of `many' vertices. By \Cref{claim:many_paths} and the definition of the ordering $\sigma$, we guarantee that there are many directed paths of length two between each pair $x, y$ with $x \in X$ and $y \in Y$. Thus if $X$ and $Y$ are large enough, we may apply induction to $T[X]$ and $T[Y]$, and then embed the remaining paths in-between greedily. To spell out the details more carefully:

\begin{lemma}\label{lem:induction-step}
There exists a partition $\cX \cup \cY = \{C_1, \ldots, C_t\}$ of 
the components such that the following holds. If $X = \bigcup \cX$ and $Y = \bigcup \cY$, then
\[
    |X\cap A_1| \ge \left(1 - \frac{1}{2\log{n}}\right)m/4\,\, \text{ and }\,\, |Y \cap A_2| \ge \left(1 - \frac{1}{2\log n}\right)m/4.
\]
\end{lemma}
\begin{proof}
Let $C^{1}_i=C_i\cap A_1$ and similarly $C^{2}_i=C_i \cap A_2$. 
Denote by $C_1,C_2,\ldots, C_{t'}$ the set of connected components for which $|C^{1}_i|\geq |C_i|/2$; this implies that $|C^{2}_{t'+\ell}|\geq |C_{t'+\ell}|/2$, for every $\ell\in [t-t']$. 
If $|\cup_{i=1}^{t'}C^{1}_i|\geq m/4$ and also $|\cup_{i=t'}^{t}C^{2}_i|\geq m/4$, then we may take the partition $\cX = \{C_i : i \in [t']\}$ and $\cY = \{C_i: i \in [t]\setminus [t']\}$. So assume that $|\cup_{i=1}^{t'}C^{1}_i|< m/4$. Choose a set $B \subset \{t'+1,\ldots, t\}$ (perhaps empty) as large as possible for which 
\[
L= \left|\bigcup_{j \in [t']\cup B} C^{1}_j \right| \leq m/4.
\]
Recall that $|C_i| \le n/5\log n$ for each $i \in [t]$. Thus if $j \in [t]\setminus (B \cup [t'])$, then the maximality of $B$ implies that 
\[
    m/4 < L + |C^1_j| \le L + |C_j|/2 \le L + \frac{n}{10\log n},
\]
and using $m = (1 - \frac{1}{5\log n})n$, we have
\begin{align*}
    L \ge m/4 - \frac{n}{10\log n} &= \left(1 - \frac{3}{5\log n}\right)n/4\\
    &\ge \left(1 - \frac{1}{2\log n}\right)\left(1 - \frac{1}{5\log n}\right)n/4\\
    &= \left(1 - \frac{1}{2\log n}\right)m/4.
\end{align*}
Moreover, by assumption we must have that $|\cup_{j\in \{t'+1,\ldots, t\}\setminus B} C^{2}_j|\geq m/4> \left(1 - \frac{1}{2\log n}\right)m/4$. Hence the partition $\cX = \{C_i : i \in [t']\cup B\}$ and $\cY = \{C_j : j \in \{t'+1, \ldots, t\} \setminus B\}$ satisfies the conclusion of the lemma.
\end{proof}

Thus \Cref{lem:induction-step} furnishes $\cX$ and $\cY$ with the stated properties. Consider $T' = T[X\cap A_1]$ and $T'' = T[Y\cap A_2]$. The following claim asserts that these subtournaments are large enough to apply induction. The proof is routine calculation. In the following, recall that $C = 10^7$.
\begin{claim}\label{claim:induction-hypothesis}
Both $T'$ and $T''$ have size at least $C\left(\frac{k}{2}\right)^2 \log^3\left(\frac{k}{2}\right)$.
\end{claim}
\begin{proof}
We prove this for $T'$; the proof for $T''$ is identical. By \Cref{lem:induction-step} we have that 
\[
    |T'| \ge \left(1 - \frac{1}{2\log{n}}\right)m/4 \ge \left(1 - \frac{1}{2\log{n}}\right)\left(1 - \frac{1}{5\log n}\right)(Ck^2/4)\log^3 k,
\]
so in order to show that $|T'| \ge C(k/2)^2\log^3(k/2)$ we must prove that 
\begin{equation}
    \left(1 - \frac{1}{2\log{n}}\right)\left(1 - \frac{1}{5\log n}\right)\log^3 k \ge \log^3\left(\frac{k}{2}\right). \label{eqn:nasty}
\end{equation}
Expanding the left-hand side, one sees that the left-hand side of (\ref{eqn:nasty}) is at least
\[
    \left(1 - \frac{7}{10\log n}\right)\log^3 k \ge \left(1 - \frac{7}{20 \log k}\right)\log^3 k,
\]
where we have used the inequality $\log n \ge \log(k^2) = 2 \log k$. Now it is not difficult to check that $(1 - 7/(20 \log k))\log^3 k \ge \log^3(k/2) = ( \log k - \log 2)^3$. Indeed, after taking cube roots and cancelling the $\log k$ factor, this is equivalent to showing that
\[
    \left(1 - \frac{7}{20\log k}\right)^{1/3} \ge 1 - \frac{\log 2}{\log k},
\]
which is easily seen to be true since $\log 2 > 7/20$. Accordingly, (\ref{eqn:nasty}) holds.
\end{proof}

So by the previous claim, $T'$ and $T''$ are large enough to find a $1$-subdivision of $T_{k/2}$ in each of them, by induction; denote the branch vertex set of each of these by $B', B''$, respectively. 
As $B'$ lies entirely before $B''$ in the ordering $\sigma$, for every $x \in B'$, $y \in B''$ we have that $d^+(x) \geq d^-(y)$.
Hence it follows from Claim~\ref{claim:many_paths} that $|N^+(x) \cap N^-(y)| \geq \binom{k}{2}+k$ and therefore we may greedily embed directed paths of length $2$ between every such pair. 
Thus we have found a $1$-subdivision of $T_k$ in $T$.
\end{proof}

\section{Concluding remarks and open problems}\label{sec:final}

Observe that our methods for embedding subdivisions of complete digraphs straightforwardly generalize to obtain the following result concerning embedding subdivisions of general digraphs in tournaments with large minimum out-degree.
\begin{theorem}\label{thm:generalization}
There exists an absolute constant $C > 0$ such that the following holds.
Let $D$ be a digraph with $m$ edges and no isolated vertices, and suppose $T$ is a tournament with $\delta^+(T) \ge Cm$. Then $T$ contains a subdivision of $D$. Moreover, each edge is subdivided at most two times.
\end{theorem}

Recall that $d(k)$ is the minimum $m$ such that any tournament $T$ with $\delta^+(T) \ge m$ contains a subdivision of $\ora{K}_k$. We have determined $d(k)$ up to a factor of $8$, and it is natural to ask whether or not the trivial lower bound is the correct answer.

\begin{question}
Is it true that $d(k) = \left(\frac{1}{4} + o(1)\right) k^2$?
\end{question}

Earlier, we mentioned that Alon, Krivelevich and Sudakov~\cite{AlonKrivSudak} proved that any graph on $n$ vertices and with at least $\varepsilon n^2$ edges contains a $1$-subdivision of a complete graph on $c(\varepsilon)\sqrt{n}$ vertices.
We conjecture that the following analogue for tournaments is true (recall that $T_k$ denotes the transitive tournament on $k$ vertices).

\begin{conjecture}\label{conj:1-subdivision}
There is a constant $C>0$ such that any tournament with at least $Ck^2$ vertices contains a $1$-subdivision of $T_k$. 
\end{conjecture}

Our Theorems~\ref{thm:trans-subdivision} and~\ref{thm:trans-1-subdivision} provide some evidence for this conjecture. Yet, it seems new ideas are needed to resolve the conjecture in full.

\bibliographystyle{amsplain}
\bibliography{tournament-subdivisions.bib}

\end{document}